\documentclass[]{amsart}

\usepackage{amssymb}
\usepackage[abbrev]{amsrefs}

\def\N{\mathbb N}
\providecommand{\abs}[1]{\left\lvert#1\right\rvert}
\providecommand{\norm}[1]{\left\lVert#1\right\rVert}
\providecommand{\ip}[2]{\left\langle #1, #2 \right\rangle}
\providecommand{\Fix}{\operatorname{F}}

\allowdisplaybreaks
\numberwithin{equation}{section}

\theoremstyle{plain}
\newtheorem{theorem}{Theorem}[section]
\newtheorem{lemma}[theorem]{Lemma}

\theoremstyle{remark}
\newtheorem{remark}[theorem]{Remark}

\title[nonexpansive mappings and quasinonexpansive mappings]{%
Approximation of fixed points of nonexpansive mappings and
quasinonexpansive mappings in a Hilbert space}

\author{Koji~Aoyama}
\address[K.~Aoyama]{%
Graduate School of Social Sciences, Chiba University, 
Yayoi-cho, Inage-ku, Chi\-ba-shi, Chiba 263-8522, Japan}
\email{aoyama@le.chiba-u.ac.jp}


\begin{document}

\begin{abstract}
 In this paper, we give a simple proof and some generalizations 
 of results in Falset et al.~\cite{MR3453020}.
\end{abstract}

\maketitle

\section{Introduction}

The aim of this paper is to give a simple proof and some generalizations 
of results in Falset, Llorens-Fuster, Marino, and
Rugiano~\cite{MR3453020}. 
In particular, we focus on the following:
\begin{theorem}[\cite{MR3453020}*{Theorem 3}] \label{t:Falset-et_al}
 Let $H$ be a Hilbert space, $C$ a nonempty closed convex subset of $H$, 
 $T\colon C\to C$ a nonexpansive mapping, $S\colon C \to C$ a strongly
 quasinonexpansive mapping, $u$ a point in $C$, 
 $\{\alpha_n\}$ and $\{\beta_n\}$ sequences in $[0,1]$,
 and $\{x_n\}$ a sequence in $C$ defined by $x_1 \in C$ and
 \begin{equation}\label{e:Falset_etal-org}
  x_{n+1} = \alpha_n u + (1-\alpha_n) \left[
      \beta_n Tx_n + (1-\beta_n) Sx_n \right]
 \end{equation}
 for $n \in \N$.
 Suppose that $F = \Fix(T) \cap \Fix(S) \ne \emptyset$, 
 $I-S$ is demiclosed at $0$, $\alpha_n \to 0$, and
 $\sum_{n=1}^\infty \alpha_n = \infty$. 
 Then the following hold: 
 \begin{enumerate}
  \item If $\sum_{n=1}^\infty (1-\beta_n) < \infty$ and
	$\sum_{n=1}^\infty \abs{\alpha_{n+1} - \alpha_n} < \infty$, then 
	$\{x_n\}$ converges strongly to $P_{\Fix(T)}(u)$; 
  \item if $\sum_{n=1}^\infty \beta_n < \infty$, $\alpha_n > 0$ for all
	$n \in \N$, and $\beta_n/\alpha_n \to 0$, then
	$\{x_n\}$ converges strongly to $P_{\Fix(S)}(u)$; 
  \item if $\liminf_n \beta_n(1-\beta_n) > 0$,
	$\alpha_n > 0$ for all $n \in \N$, then
	$\{x_n\}$ converges strongly to $P_F (u)$. 
 \end{enumerate}
\end{theorem}

\begin{remark}
 The assumption that ``$\alpha_n > 0$ for all $n \in \N$'' is not
 described in~(2) and~(3) of \cite{MR3453020}*{Theorem~3}.
 However, this assumption is used in the proof of the theorem. 
\end{remark}

This paper is organized as follows: 
In the next section, some notions and definitions are introduced.
In \S3, we give a simple proof of (1) of Theorem~\ref{t:Falset-et_al}. 
In \S4, we obtain a slight generalization
(Theorem~\ref{t:g=Falset-et_al(2)}) of (2) of
Theorem~\ref{t:Falset-et_al}.
In \S5, we prove a strong convergence theorem for two quasinonexpansive
mappings (Theorem~\ref{t:general3}),
which is a generalization of (3) of Theorem~\ref{t:Falset-et_al}.

\section{Preliminaries}
Throughout the present paper, 
$H$ denotes a real Hilbert space, 
$\ip{\,\cdot\,}{\,\cdot\,}$ the inner product of $H$, 
$\norm{\,\cdot\,}$ the norm of $H$, 
$C$ a nonempty closed convex subset of $H$, 
$I$ the identity mapping on $H$, 
and $\N$ the set of positive integers. 
Strong convergence of a sequence $\{x_n\}$ in $H$ to $x\in H$ is denoted
by $x_n \to x$ and weak convergence by $x_n \rightharpoonup x$. 

It is known that
\begin{equation}\label{e:2uc2us}
 \norm{\lambda x + (1-\lambda)y}^2 
  = \lambda \norm{x}^2 + (1-\lambda) \norm{y}^2 
  - \lambda (1-\lambda) \norm{x-y}^2
\end{equation}
holds for all $x,y \in H$ and a real number $\lambda$. 

Let $T\colon C\to H$ be a mapping. 
The set of fixed points of $T$ is denoted by $\Fix(T)$. 
A mapping $T$ is said to be \emph{quasinonexpansive} if 
$\Fix(T) \ne \emptyset$ and $\norm{Tx - p} \leq \norm{x-p}$ 
for all $x\in C$ and $p\in \Fix(T)$; 
$T$ is said to be \emph{nonexpansive} if 
$\norm{Tx - Ty} \leq \norm{x-y}$ for all $x,y\in C$; 
$T$ is said to be \emph{strongly quasinonexpansive}~\cites{%
1703.02218,pNACA2015,MR3213161,%
MR3258665,MR2529497,MR2884574,MR2960628,%
MR2358984,MR1386686}
if each $T$ is quasinonexpansive and $Tx_n - x_n \to 0$ whenever
$\{x_n\}$ is a bounded sequence in $C$ and $\norm{x_n - p} - \norm{Tx_n
- p} \to 0$ for some point $p\in \Fix(T)$;
$T$ is \emph{demiclosed at $0$} if 
$Tp = 0$ whenever $\{ x_n \}$ is a sequence in $C$ such that 
$x_n\rightharpoonup p$ and $T x_n \to 0$. 
We know the following:
\begin{itemize}
 \item If $T\colon C\to H$ is quasinonexpansive, then $\Fix(T)$ is
       closed and convex; see~\cite{MR0298499}*{Theorem~1};
 \item if $T\colon C\to H$ is nonexpansive,
       then $I-T$ is demiclosed at $0$; see~\cite{MR1074005}. 
\end{itemize}

\begin{remark}
 Let $T\colon C\to H$ be a mapping with a fixed point. 
 Then $T$ is strongly quasinonexpansive if and only if 
 $T$ is of type~(sr) in the sense of \cites{%
 MR3258665,MR2529497,MR2884574,MR2960628}. 
\end{remark}

The following lemma is essentially proved in~\cite{MR0324491};
see~\cite{MR0324491}*{Lemma 3} and its proof. 
We give the proof for the sake of completeness. 

\begin{lemma} \label{l:convex_T-S}
 Let $T\colon C \to H$ and $S\colon C\to H$ be quasinonexpansive mappings
 and $U \colon C\to H$ a mapping defined by 
 $U = \beta T + (1-\beta) S$, where $\beta \in (0,1)$. 
 Suppose that $\Fix(T) \cap \Fix(S) \ne \emptyset$. 
 Then $\Fix(U) = \Fix(T) \cap \Fix(S)$ and $U$ is quasinonexpansive. 
\end{lemma}

\begin{proof}
 Set $F=\Fix(T) \cap \Fix(S)$. 
 Clearly, $\Fix(U) \supset F$. Thus $\Fix(U)$ is nonempty. 
 We show $\Fix(U) \subset F$.
 Let $z \in \Fix(U)$ and $w \in F$ be given. 
 Since $T$ and $S$ are quasinonexpansive and $w \in \Fix(T) \cap \Fix
 (S)$, it follows from~\eqref{e:2uc2us} that 
 \begin{align*}
  \norm{z-w}^2 &= \norm{Uz -w}^2
  = \norm{\beta (Tz - w) + (1-\beta) (Sz-w)}^2 \\
  &= \beta \norm{Tz - w}^2 + (1-\beta)\norm{Sz-w}^2
  - \beta(1-\beta)\norm{Tz-Sz}^2 \\
  &\leq \norm{z - w}^2  - \beta(1-\beta)\norm{Tz-Sz}^2. 
 \end{align*}
 Thus $\norm{Tz - Sz} =0$ and hence $Tz=Sz$.
 This implies that $z \in F$. Therefore,
 $\Fix(U) \subset F$, and thus $\Fix(U)  = F$. 
 It is easy to verify that $U$ is quasinonexpansive. 
\end{proof}

It is known that, for each $x \in H$, 
there exists a unique point $x_0 \in C$ such that 
\[
 \norm{x - x_0} = \min\{\norm{x-y}: y\in C\}. 
\]
Such a point $x_0$ is denoted by $P_C (x)$ and $P_C$ is called the
\emph{metric projection} of $H$ onto $C$; see~\cite{MR2548424}
for more details. 

Let $\{T_n\}$ be a sequence of mappings of $C$ into $H$ 
such that $F=\bigcap_{n=1}^\infty \Fix(T_n)$ is nonempty. Then 
$\{T_n\}$ is said to be \emph{strongly quasinonexpansive type}~\cites{%
MR3258665,MR2960628,MR2529497,MR3013135,%
NMJ2016,MR3213161} if each $T_n$ is quasinonexpansive and 
$T_n x_n - x_n \to 0$ whenever $\{x_n\}$ is a bounded sequence
in $C$ and $\norm{x_n - p} - \norm{T_n x_n - p} \to 0$ 
for some point $p\in F$; 
$z \in C$ is said to be an \emph{asymptotic fixed point} of
$\{T_n\}$ if
there exist a sequence $\{ x_n \}$ in $C$ and a subsequence
$\{x_{n_i}\}$ of $\{x_n\}$ such that
$T_n x_n - x_n \to 0$ and $x_{n_i} \rightharpoonup z$;
see~\cite{MR3013135}. 
The set of asymptotic fixed points of $\{T_n\}$ is denoted by
$\hat{\Fix}(\{T_n\})$. It is clear that $F \subset \hat{\Fix}(\{T_n\})$.
The following lemma is also clear from the definition: 

\begin{lemma}\label{l:T_n=T}
 Let $T\colon C\to H$ be a mapping with a fixed point.
 Suppose that $T_n = T$ for $n\in \N$. Then the following hold:
 \begin{itemize}
  \item If $I-T$ is demiclosed at $0$,
	then $\hat{\Fix}(\{T_n\}) = \bigcap_{n=1}^\infty \Fix(T_n)$;
  \item if $T$ is strongly quasinonexpansive, 
	then $\{T_n\}$ is strongly quasinonexpansive type. 
 \end{itemize}
\end{lemma}

\begin{remark}\label{r:sqnt=srns:F^=(Z)}
 Let $\{T_n\}$ be a sequence of mappings of $C$ into $H$
 such that $F=\bigcap_{n=1}^\infty \Fix(T_n)$ is nonempty. 
 Then we know the following: 
 \begin{itemize}
  \item  $\{T_n\}$ is strongly quasinonexpansive type
	 if and only if it is a strongly relatively nonexpansive sequence 
	 in the sense of \cites{%
	 MR3258665,MR2960628,MR2529497,MR3013135}; 
	 see~\cite{MR3213161}*{Remark~2.5};
  \item $F = \hat{\Fix}(\{T_n\})$
	if and only if $\{T_n\}$ satisfies the condition~(Z)
	in the sense of~\cites{MR3258665,MR3185784,MR3213161,MR2960628,%
	MR2799767,MR2671943,MR2529497,
	MR3203624,pNACA2011,pNLMUA2011,MR2762191,MR2762173,MR3013135};
	see~\cite{MR3013135}*{Proposition~6};
  \item if $z \in \hat{\Fix}(\{T_n\})$, then 
	there exist a \emph{bounded} sequence $\{ x_n \}$ in $C$ and 
	a subsequence $\{x_{n_i}\}$ of $\{x_n\}$ such that
	$T_n x_n - x_n \to 0$ and $x_{n_i} \rightharpoonup z$;
	see the proof of~\cite{MR3013135}*{Proposition~6}. 
 \end{itemize}
\end{remark}

The following lemma is well known; see, for example,
\cite{MR1353071}*{Lemma 2}. 

\begin{lemma}\label{lm:seq}
 Let $\{ \xi_n \}$ be a sequence of nonnegative real numbers, 
 $\{ \alpha_n \}$ a sequence in $[0,1]$, 
 and $\{ \gamma_n \}$ a sequence of nonnegative real numbers. 
 Suppose that 
 $\xi_{n+1} \leq (1- \alpha_n) \xi_n + \gamma_n$
 for all $n\in\N$, 
 $\sum_{n=1}^\infty \alpha_n = \infty$,
 and $\sum_{n=1}^\infty \gamma_n < \infty$.
 Then $\xi_n \to 0$. 
\end{lemma}

\section{Proof of Theorem~\ref{t:Falset-et_al}~(1)}

In this section, we give a simple proof of the conclusion~(1) of
Theorem~\ref{t:Falset-et_al}. The proof depends heavily on the
following well known result; see~\cite{MR1156581}*{Theorem 2} and
\cite{MR1402593}*{Theorem 3.2}.

\begin{theorem}\label{t:wittmann}
 Let $H$ be a Hilbert space, $C$ a nonempty closed convex subset of $H$, 
 $T\colon C\to C$ a nonexpansive mapping, $u$ a point in $C$, 
 $\{\alpha_n\}$ a sequence in $[0,1]$,
 and $\{y_n\}$ a sequence in $C$ defined by $y_1 \in C$ and
 \begin{equation}\label{e:halpern}
  y_{n+1} = \alpha_n u + (1-\alpha_n) Ty_n  
 \end{equation}
 for $n \in \N$. 
 Suppose that $\Fix(T)$ is nonempty, 
 $\alpha_n \to 0$, $\sum_{n=1}^\infty \alpha_n = \infty$, and
 $\sum_{n=1}^\infty \abs{\alpha_{n+1} - \alpha_n} < \infty$.
 Then $\{y_n\}$ converges strongly to $P_{\Fix(T)}(u)$. 
\end{theorem}

The following lemma connects Theorem~\ref{t:wittmann} and 
Theorem~\ref{t:Falset-et_al}~(1). 

\begin{lemma}\label{l:y_n2w=>x_n2w}
 Let $H$, $C$, $T$, and $u$ be the same as in Theorem~\ref{t:wittmann}, 
 $\{\alpha_n\}$ and $\{\beta_n\}$ sequences in $[0,1]$,
 $\{z_n\}$ a bounded sequence in $C$, 
 $\{x_n\}$ a sequence in $C$ defined by $x_1 \in C$ and
 \[
 x_{n+1} = \alpha_n u + (1-\alpha_n) \left[
 \beta_n Tx_n + (1-\beta_n) z_n \right] 
 \]
 for $n \in \N$, and $\{y_n\}$ a sequence in $C$ defined by $y_1 \in C$ 
 and~\eqref{e:halpern} for $n \in \N$. 
 Suppose that $\sum_{n=1}^\infty \alpha_n = \infty$ and 
 $\sum_{n=1}^\infty (1-\beta_n) < 0$. Then $x_n - y_n \to 0$. 
\end{lemma}
\begin{proof}
 Let $v \in \Fix(T)$ be fixed.  
 Since $T$ is nonexpansive, we deduce that
 \begin{align*}
  \norm{y_{n+1} -v}
  &\leq \alpha_n \norm{u-v} + (1-\alpha_n)\norm{Ty_n -v} \\
  &\leq \alpha_n \norm{u-v} + (1-\alpha_n)\norm{y_n - v}. 
 \end{align*}
 Thus, by induction on $n$, we see that 
 \[
 \norm{Ty_{n+1} -v} \leq \norm{y_{n+1} -v}   
 \leq \max\{ \norm{u-v}, \norm{y_1 -v}\}  
 \]
 for all $n \in \N$. Hence $\{ Ty_n\}$ is bounded.
 Since $\{Ty_n\}$ and $\{z_n\}$ are bounded, it follows that
 $\sup_n \norm{z_n - Ty_n} < \infty$. 
 Thus we have
 \begin{align*}
  \norm{x_{n+1} - y_{n+1}}
  &\leq (1-\alpha_n) \bigl[ \beta_n \norm{Tx_n - Ty_n}
  + (1- \beta_n) \norm{z_n - Ty_n}\bigr] \\
  &\leq (1-\alpha_n)\norm{x_n - y_n}
  + (1- \beta_n) \sup_n \norm{z_n - Ty_n}
 \end{align*}
 for all $n \in \N$. 
 Therefore, $\norm{x_n - y_n} \to 0$ by Lemma~\ref{lm:seq}. 
\end{proof}

Using Theorem~\ref{t:wittmann} and Lemma~\ref{l:y_n2w=>x_n2w},
we can prove Theorem~\ref{t:Falset-et_al}~(1) as follows: 

\begin{proof}[Proof of Theorem~\ref{t:Falset-et_al}~(1)]
 Let $v \in F$ be fixed.  
 Since $T$ is nonexpansive, $S$ is quasinonexpansive,
 and $v \in \Fix(T)\cap \Fix(S)$, we have
 \begin{align}
  \begin{split}\label{e:bounded}
   \norm{x_{n+1} -v}
   &\leq \alpha_n \norm{u-v} + (1-\alpha_n)
   [\beta_n \norm{Tx_n -v} + (1-\beta_n) \norm{Sx_n-v}] \\
   &\leq \alpha_n \norm{u-v} + (1-\alpha_n) \norm{x_n -v}.
  \end{split}
 \end{align}
 Thus, by induction on $n$, we see that 
 \begin{equation}\label{e:bounded2}
 \norm{Sx_{n+1} -v} \leq \norm{x_{n+1} -v}   
 \leq \max\{ \norm{u-v}, \norm{x_1 -v}\}
 \end{equation}
 for all $n \in \N$. Hence $\{ Sx_n\}$ is bounded.
 Let $\{y_n\}$ be a sequence defined by $y_1 \in C$
 and~\eqref{e:halpern} for $n \in \N$. 
 Then Theorem~\ref{t:wittmann} implies that $y_n \to P_{\Fix(T)}(u)$ and
 Lemma~\ref{l:y_n2w=>x_n2w} implies that $x_n - y_n\to 0$.  
 Therefore we conclude that $x_n \to P_{\Fix(T)}(u)$. 
\end{proof}

\section{A generalization of Theorem~\ref{t:Falset-et_al}~(2)}

In this section, we prove the following theorem, which is a slight
generalization of the conclusion~(2) of Theorem~\ref{t:Falset-et_al}. 
In fact, the condition $\sum_{n=1}^\infty \beta_n < \infty$
can be replaced by $\beta_n \to 0$.

\begin{theorem}\label{t:g=Falset-et_al(2)}
 Let $H$ be a Hilbert space, $C$ a nonempty closed convex subset of $H$, 
 $T\colon C\to C$ a nonexpansive mapping, $S\colon C \to C$ a strongly
 quasinonexpansive mapping, $u$ a point in $C$, 
 $\{\alpha_n\}$ a sequence in $(0,1]$,
 $\{\beta_n\}$ a sequence in $[0,1]$,
 and $\{x_n\}$ a sequence in $C$ defined by $x_1 \in C$ and
 \eqref{e:Falset_etal-org} for $n \in \N$.
 Suppose that $F = \Fix(T) \cap \Fix(S)$ is nonempty, 
 $I-S$ is demiclosed at $0$, $\alpha_n \to 0$,
 $\sum_{n=1}^\infty \alpha_n = \infty$, 
 $\beta_n \to 0$, and $\beta_n/\alpha_n \to 0$. 
 Then $\{x_n\}$ converges strongly to $P_{\Fix(S)}(u)$. 
\end{theorem}

The proof of Theorem~\ref{t:g=Falset-et_al(2)} essentially depends on
the following theorem, which is a direct consequence
of~\cite{NMJ2016}*{Corollary 3.1} and Lemma~\ref{l:T_n=T}. 

\begin{theorem}\label{t:aoyama-nmj2016}
 Let $H$, $C$, $S$, and $\{\alpha_n\}$ be the same as in
 Theorem~\ref{t:g=Falset-et_al(2)}, $\{u_n\}$ a sequence in $C$, 
 and $\{x_n\}$ a sequence in $C$ defined by $x_1 \in C$ and
 \[
 x_{n+1} = \alpha_n u_n + (1-\alpha_n) Sx_n 
 \]
 for $n \in \N$. Suppose that $u_n \to u$. 
 Then $\{x_n\}$ converges strongly to $P_{\Fix(S)}(u)$. 
\end{theorem}

Using Theorem~\ref{t:aoyama-nmj2016}, we obtain the following lemma:

\begin{lemma}\label{l:u_n}
 Let $H$, $C$, $S$, $u$, $\{\alpha_n\}$, and $\{\beta_n\}$ be the same
 as in Theorem~\ref{t:g=Falset-et_al(2)},
 $\{z_n\}$ a bounded sequence in $C$, 
 and $\{x_n\}$ a sequence in $C$ defined by $x_1 \in C$ and
 \begin{equation}\label{162921_16Aug17}
  x_{n+1} =
   \alpha_n u + (1-\alpha_n) [\beta_n z_n + (1-\beta_n) Sx_n ]
 \end{equation}
 for $n \in \N$. 
 Then $\{x_n\}$ converges strongly to $P_{\Fix(S)}(u)$. 
\end{lemma}

\begin{proof} 
 Set $\gamma_n = \alpha_n + \beta_n - \alpha_n \beta_n$. 
 Then it is easy to verify that $0<\gamma_n \leq 1$ for all $n \in \N$,
 $\gamma_n \to 0$, and $\sum_{n=1}^\infty \gamma_n = \infty$.
 From the definition of $\{x_n\}$, the equation~\eqref{162921_16Aug17} is
 reduced to  
 \[
 x_{n+1} = \gamma_n u_n + (1-\gamma_n) Sx_n, 
 \]
 where $u_n = [\alpha_n u + (1-\alpha_n) \beta_n z_n]/ \gamma_n$. 
 Since $\beta_n/\alpha_n \to 0$ and $\{z_n\}$ is bounded,
 we see that $u_n \to u$. 
 Therefore Theorem~\ref{t:aoyama-nmj2016} implies the conclusion. 
\end{proof}

Using Lemma~\ref{l:u_n}, we can prove
Theorem~\ref{t:g=Falset-et_al(2)} as follows:

\begin{proof}[Proof of Theorem~\ref{t:g=Falset-et_al(2)}]
 Let $v \in F$ be fixed. Since $T$ is nonexpansive and $v \in \Fix(T)$, 
 it follows from~\eqref{e:bounded} and~\eqref{e:bounded2} that
 \[
 \norm{Tx_{n+1} -v} \leq \norm{x_{n+1} -v}   
 \leq \max\{ \norm{u-v}, \norm{x_1 -v}\}
 \]
 for all $n \in \N$. Thus $\{ Tx_n \}$ is bounded. 
 Hence Lemma~\ref{l:u_n} implies that $x_n \to P_{\Fix(T)}(u)$. 
\end{proof}

\section{Strong convergence theorem for two quasinonexpansive mappings}

In this section, we prove the following strong convergence theorem for
two quasinonexpansive mappings. 

\begin{theorem}\label{t:general3}
 Let $H$ be a Hilbert space, $C$ a nonempty closed convex subset of $H$, 
 $T\colon C\to C$ a quasinonexpansive mapping, 
 $S\colon C \to C$ a strongly quasinonexpansive mapping, 
 $u$ a point in $C$, $\{\alpha_n\}$ a sequence in $(0,1]$, 
 $\{\beta_n\}$ a sequence in $[0,1]$,
 and $\{x_n\}$ a sequence in $C$ defined by $x_1 \in C$ and
 \eqref{e:Falset_etal-org} for $n \in \N$. 
 Suppose that $F = \Fix(T) \cap \Fix(S)$ is nonempty, 
 both $I-T$ and $I-S$ are demiclosed at $0$,
 $\alpha_n \to 0$, $\sum_{n=1}^\infty \alpha_n = \infty$, 
 and $\liminf_n \beta_n(1-\beta_n) > 0$. 
 Then $\{x_n\}$ converges strongly to $P_F (u)$. 
\end{theorem}

\begin{remark}
 Under the assumptions of Theorem~\ref{t:Falset-et_al}~(3),
 $T$ is quasinonexpansive and $I-T$ is demiclosed at $0$. 
 Therefore Theorem~\ref{t:Falset-et_al}~(3) is a corollary of
 Theorem~\ref{t:general3}. 
\end{remark}

The proof of Theorem~\ref{t:general3} is based on the following theorem, 
which is derived from~\cite{MR2960628}*{Theorem~3.1} and
Remark~\ref{r:sqnt=srns:F^=(Z)}. 

\begin{theorem}\label{t:AKK2011}
 Let $H$, $C$, $u$, and $\{\alpha_n\}$ be the same as in
 Theorem~\ref{t:general3}, 
 $\{U_n\}$ a sequence of mappings of $C$ into itself, 
 and $\{x_n\}$ a sequence in $C$ defined by $x_1 \in C$ and
 \[
 x_{n+1} = \alpha_n u + (1-\alpha_n) U_n x_n
 \]
 for $n \in \N$. 
 Suppose that $K = \bigcap_{n=1}^\infty \Fix(U_n) \ne \emptyset$,
 $\{U_n\}$ is strongly quasinonexpansive type,
 and $\hat{\Fix}(\{U_n\}) = K$. 
 Then $\{x_n\}$ converges strongly to $P_K (u)$. 
\end{theorem}

In order to prove Theorem~\ref{t:general3}, we need the following
lemma. For similar results, see \cites{%
MR2377867,MR2581778,MR2884574,MR2529497,MR0470761}. 

\begin{lemma}\label{l:U_n:sqnt-(Z)}
 Let $H$, $C$, $T$, $S$, and $\{\beta_n\}$ be the same as in
 Theorem~\ref{t:general3} and
 $U_n \colon C\to C$ a mapping defined by
 \[
  U_n = \beta_n T + (1-\beta_n) S
 \]
 for $n \in \N$. Then the following hold:
 \begin{enumerate}
  \item Each $U_n$ is quasinonexpansive and
	$\Fix(T) \cap \Fix(S) = \bigcap_{n=1}^\infty \Fix(U_n)$; 
  \item $\{U_n\}$ is strongly quasinonexpansive type;
  \item $\hat{\Fix} (\{ U_n \}) = \bigcap_{n=1}^\infty \Fix(U_n)$. 
 \end{enumerate}
\end{lemma}

\begin{proof}
 By assumption, there exist $a,b \in (0,1)$ and $m \in \N$ such that
 $a\leq \beta_n \leq b$ for all $n \geq m$. 

 We first show (1). Set $F = \Fix(T) \cap \Fix(S)$. 
 Lemma~\ref{l:convex_T-S} implies that
 $U_n$ is quasinonexpansive and $\Fix(U_n) = F$ for all $n \geq m$, 
 and hence $U_n$ is quasinonexpansive and $\Fix(U_n) \supset F$ for all
 $n \in \N$. Thus we have $F \subset \bigcap_{n=1}^\infty \Fix(U_n) 
 \subset \bigcap_{n=m}^\infty \Fix(U_n) = F$. 
 Therefore, $F = \bigcap_{n=1}^\infty \Fix(U_n)$.  

 We next show (2). Suppose that $\{x_n\}$ is a bounded sequence in $C$
 and $\norm{x_n - p} - \norm{U_n x_n -p}\to 0$ for some
 $p \in \bigcap_{n=1}^\infty \Fix(U_n)$. 
 From~(1), it is enough to show that $x_n - U_n x_n \to 0$. 
 Since $S$ and $T$ are quasinonexpansive and $p \in F$ by (1), we have
 \begin{align*}
  \norm{U_n x_n - p}
  &\leq \beta_n \norm{Tx_n -p} + (1-\beta_n) \norm{Sx_n -p }\\
  &\leq \beta_n \norm{x_n -p} + (1-\beta_n) \norm{Sx_n -p }
  \leq \norm{x_n -p}. 
 \end{align*}
 Hence we see that
 \begin{equation}\label{161425_30Aug17}
  0 \leq (1-b)(\norm{x_n -p} - \norm{Sx_n - p})
 \leq \norm{x_n -p} - \norm{U_n x_n - p}
 \end{equation}
 for all $n \geq m$. 
 Since $\norm{x_n -p} - \norm{Sx_n - p} \to 0$
 by~\eqref{161425_30Aug17}, 
 $S$ is strongly quasinonexpansive, and $p \in \Fix(S)$ by~(1), 
 we conclude that
 \begin{equation}\label{e:x_n-Sx_n-to-0}
  \norm{x_n - Sx_n} \to 0. 
 \end{equation}
 Using~\eqref{e:2uc2us} and (1), we have
 \begin{align*}
  \beta_n (1-\beta_n) \norm{Tx_n -Sx_n}^2
  &= \beta_n \norm{Tx_n -p}^2 + (1-\beta_n) \norm{Sx_n -p}^2 
  - \norm{U_n x_n - p}^2\\
  &\leq \norm{x_n -p}^2 - \norm{U_n x_n - p}^2 \\
  &\leq 2\norm{x_n -p} (\norm{x_n -p}  - \norm{U_n x_n - p}), 
 \end{align*}
 which implies that
 \[
 0\leq a(1-b) \norm{Tx_n -Sx_n}^2
 \leq 2\norm{x_n -p} (\norm{x_n -p}  - \norm{U_n x_n - p})
 \]
 for all $n \geq m$. 
 Taking into account the boundedness of $\{x_n\}$, we see that
 \begin{equation}\label{e:Tx_n-Sx_n-to-0}
  \norm{Tx_n - Sx_n} \to 0. 
 \end{equation}
 Therefore it follows from~\eqref{e:x_n-Sx_n-to-0}
 and~\eqref{e:Tx_n-Sx_n-to-0} that 
 \[
 \norm{x_n - U_n x_n}
 \leq \norm{x_n - Sx_n} + \beta_n \norm{Sx_n - Tx_n} \to 0.
 \]
 Consequently, $\{U_n\}$ is strongly quasinonexpansive type. 

 We next show (3). From (1), it is enough to show that $\hat{\Fix}(\{
 U_n \}) \subset F$, where $F = \Fix(T) \cap \Fix(S)$. 
 Let $z \in \hat{\Fix}(\{ U_n \})$.
 Then there exist a bounded sequence $\{ x_n \}$ in $C$ and
 a subsequence $\{x_{n_i}\}$ of $\{x_n\}$ such that
 $U_n x_n - x_n \to 0$ and $x_{n_i} \rightharpoonup z$; 
 see Remark~\ref{r:sqnt=srns:F^=(Z)}. 
 Let $p \in F$. As in the proof of~(2), it turns out that
 \[
 0 \leq (1-b)(\norm{x_n -p} - \norm{Sx_n - p})
 \leq \norm{x_n -p} - \norm{U_n x_n - p}
 \leq \norm{x_n - U_n x_n}
 \]
 for all $n \geq m$, 
 and hence \eqref{e:x_n-Sx_n-to-0} holds. 
 Since $I-S$ is demiclosed at $0$, we conclude that $z \in \Fix(S)$. 
 On the other hand, by the definition of $U_n$, we see that 
 \begin{align*}
  0 \leq a \norm{x_n - Tx_n} &\leq \beta_n \norm{x_n - Tx_n}
  = \norm{x_n - U_n x_n - (1-\beta_n) (x_n - Sx_n)}\\
  &\leq \norm{x_n - U_n x_n} + \norm{x_n - Sx_n}
 \end{align*}
 for all $n \geq m$. 
 Taking into account $U_n x_n -x_n \to 0$ and \eqref{e:x_n-Sx_n-to-0}, 
 we deduce that $\norm{x_n - Tx_n} \to 0$. 
 Since $I-T$ is demiclosed at $0$, we conclude that $z \in \Fix(T)$. 
 Consequently, $z \in \Fix(T) \cap \Fix(S) = F$. 
 This completes the proof.
\end{proof}

Using Theorem~\ref{t:AKK2011} and Lemma~\ref{l:U_n:sqnt-(Z)}, we can
prove Theorem~\ref{t:general3} as follows: 

\begin{proof}[Proof of Theorem~\ref{t:general3}]
 Set $U_n = \beta_n T + (1-\beta_n)S$ for $n \in \N$.
 Then Lemma~\ref{l:U_n:sqnt-(Z)} shows that
 $F = \bigcap_{n=1}^\infty \Fix(U_n)$,
 $\{U_n\}$ is strongly quasinonexpansive type,
 and $\hat{\Fix}(\{U_n\}) = F$. 
 Thus Theorem~\ref{t:AKK2011} implies the conclusion.
\end{proof}


\end{document}